\documentclass[12pt,a4paper]{amsart}
\usepackage{amsfonts}
\usepackage{amsthm}
\usepackage{amsmath}
\usepackage{mathrsfs}
\usepackage{amscd}
\usepackage[latin2]{inputenc}
\usepackage{t1enc}
\usepackage{graphicx}
\usepackage{graphics}
\usepackage{pict2e}
\usepackage{epic}
\numberwithin{equation}{section}
\usepackage[margin=2.9cm]{geometry}
\usepackage{epstopdf} 
\usepackage{tikz}
\tikzstyle{v}=[circle, draw, fill=red, inner sep=0pt, minimum width=4pt]
\tikzstyle{w}=[circle, draw, fill=blue, inner sep=0pt, minimum width=4pt]
\tikzset{>=stealth}

\usepackage{enumerate}

\theoremstyle{plain}
\newtheorem{Th}{Theorem}[section]
\newtheorem{Lemma}[Th]{Lemma}
\newtheorem{Cor}[Th]{Corollary}
\newtheorem{Prop}[Th]{Proposition}

 \theoremstyle{definition}
\newtheorem{Def}[Th]{Definition}

\newtheorem{?}[Th]{Question}
\newtheorem{Ex}[Th]{Example}

\newcommand{\NN}{\mathbf N}
\newcommand{\ZZ}{\mathbf Z}
\newcommand{\del}{\partial}
\newcommand{\paths}{\mathscr P}

\begin{document}

\title{Determinantal formulas for SEM Expansions of Schubert Polynomials}


\author{Hassan Hatam}
\address{North Carolina State University \\ Department of Mathematics \\
Raleigh NC 27695} 
\email{hhatam@ncsu.edu}

\author{Joseph Johnson}
\address{North Carolina State University \\ Department of Mathematics \\
Raleigh NC 27695} 
\email{jwjohns5@ncsu.edu}

\author{Ricky Ini Liu}
\address{North Carolina State University \\ Department of Mathematics \\
	Raleigh NC 27695} 
\email{riliu@ncsu.edu}

\author{Maria Macaulay}
\address{North Carolina State University \\ Department of Mathematics \\
Raleigh NC 27695} 
\email{mlmacaul@ncsu.edu}

\thanks{The authors were partially supported by a National Science Foundation grant DMS-1700302. R.~I.~Liu was also partially supported by a National Science Foundation grant CCF-1900460.}



\begin{abstract}
	We show that for any permutation $w$ that avoids a certain set of 13 patterns of lengths 5 and 6, the Schubert polynomial $\mathfrak S_w$ can be expressed as the determinant of a matrix of elementary symmetric polynomials in a manner similar to the Jacobi-Trudi identity. For such $w$, this determinantal formula is equivalent to a (signed) subtraction-free expansion of $\mathfrak S_w$ in the basis of standard elementary monomials.
\end{abstract}

\maketitle

\section{Introduction}

The Schubert polynomials $\mathfrak S_w$ form an important basis of the polynomial ring $\ZZ[x_1, x_2, \dots]$, primarily due to their role as representatives for the classes of Schubert varieties in the cohomology of the flag variety. In this paper, we consider the expansion of Schubert polynomials in the SEM basis consisting of \emph{standard elementary monomials}  \[e_{j_1j_2\cdots}=e_{j_1}(x_1)e_{j_2}(x_1, x_2)e_{j_3}(x_1, x_2, x_3)\cdots,\]
where $e_{k}$ is the $k$th elementary symmetric polynomial and only finitely many of the $j_i$ are nonzero. Such SEM expansions of Schubert polynomials have been studied previously in \cite{FGP, Kirillov, LS, PS, Winkel}. 
In particular, it was shown by Fomin, Gelfand, and Postnikov \cite{FGP} that these expansions are important for the construction of \emph{quantum Schubert polynomials}, which can be used to compute Gromov-Witten invariants for the small quantum cohomology ring of the flag variety. Additionally, Postnikov and Stanley \cite{PS} noted that the problem of finding the SEM expansion of Schubert polynomials is equivalent to the problem of computing the \emph{inverse Schubert-Kostka matrix}---that is, the expansion of monomials in the Schubert basis.

One special case of Schubert polynomials are the Schur polynomials $s_\lambda(x_1, \dots, x_n)$, which have a determinantal formula in terms of elementary symmetric polynomials via the famous Jacobi-Trudi identity. It was observed by Kirillov \cite{Kirillov} that this identity can be slightly modified to give a determinantal formula that, when expanded, gives the SEM expansion for Schur polynomials. (See Corollary~\ref{cor:grassmannian} below.) A similar determinantal formula was given in \cite{PS} for $\mathfrak S_w$ when $w$ is a $213$-avoiding permutation. Such determinants can be interpreted via a nonintersecting lattice path model using the Lindstr\"om-Gessel-Viennot lemma.

Our main focus will be to study which Schubert polynomials $\mathfrak S_w$ can be expressed as a Jacobi-Trudi-like determinant that yields its SEM expansion (and can therefore be described by a nonintersecting lattice path model). Such determinantal formulas are particularly notable because any coefficient appearing in such an SEM expansion has absolute value at most $1$. Our main result will be to show that such a determinantal formula exists when $w$ avoids the following 13 patterns of length 5 and 6:
\[	51324, \quad
15324, \quad
52413, \quad
25413, \quad
53142, \quad
35142, \quad
31542,\]\[
143265, \quad
143625, \quad
143652, \quad
146352, \quad
413265, \quad
413625.\]
(This is not a necessary condition---see \S5 for further discussion.)


Our approach will utilize the fact that certain operations such as divided difference operators can be seen to act on the generating functions for nonintersecting lattice paths by moving the endpoints in a simple combinatorial way. A similar observation was also used in \cite{CLL, CYY} to give lattice path interpretations for certain flagged double Schur functions and flagged skew Schubert polynomials (though the interpretations there primarily yield formulas in terms of complete homogeneous symmetric polynomials rather than elementary symmetric polynomials).

The organization of this paper is as follows: In \S2, we will discuss background information on permutations, Schubert polynomials, and standard elementary monomials, as well as define lattice path representations for polynomials. We will also discuss how these lattice path models apply to the context of quantum Schubert polynomials. In \S3, we will discuss various operations for manipulating lattice path representations. In \S4, we will use the operations in \S3 to first prove a special case regarding $1324$-avoiding separable permutations and then build on this case to prove our main result in Theorem~\ref{thm:main}. We will conclude in \S5 with some remaining open questions.

\section{Background}

In this section, we will introduce necessary background about permutations, Schubert polynomials, standard elementary monomials, and nonintersecting lattice paths. For more information, see, for instance, \cite{Manivel}.


\subsection{Permutations}

Let $S_n$ denote the symmetric group of permutations on $[n] = \{1, \dots, n\}$. We will often denote a permutation $w\in S_n$ in one-line notation $w=w_1w_2 \cdots w_n$.

The simple transpositions $s_i = (i \;\; i+1)$ for $i = 1, \dots, n-1$ generate the group $S_n$. For a permutation $w \in S_n$, its \emph{length} $\ell(w)$ is the length of the shortest expression for $w$ as a product of simple transpositions $s_{i_1} \dots s_{i_\ell}$ (called a \emph{reduced expression}). Alternatively, $\ell(w)$ is the number of \emph{inversions} of $w$, where an inversion is an ordered pair $(w_i,w_j)$ satisfying $j>i$ and $w_j < w_i$. We denote by $1_n$ the identity permutation in $S_n$, and we denote by $w_0 = w_0^{(n)}$ the permutation $n (n-1)\cdots 1 \in S_n$ of maximum length in $S_n$.

The \emph{(Lehmer) code} of a permutation $w \in S_n$ is the sequence $c=(c_1, \dots, c_n)$, where $c_i = \#\{j > i \mid w_j < w_i\}$. The map from $w \in S_n$ to its code $c$ is a bijection from $S_n$ to the set of integer vectors $(c_1, \dots, c_n)$ satisfying $0 \leq c_i \leq n-i$ for all $i$.

For a permutation $w$, we say that $w_i$ is a \emph{left-to-right maximum} of $w$ if $w_j < w_i$ for all $j < i$.

Sometimes it will be convenient to consider the direct limit $S_\infty$ of symmetric groups under the natural embeddings $\iota\colon S_n \hookrightarrow S_{n+1}$ in which $S_n$ acts on the first $n$ letters. Equivalently, any element $w \in S_\infty$ is a permutation of $\NN = \{1, 2, \dots\}$ that fixes all but finitely many elements.

\subsubsection{Pattern avoidance}
Given a permutation (or \emph{pattern}) $p = p_1 \cdots p_k \in S_k$, we say that a permutation $w \in S_n$ \emph{contains} the pattern $p$ if $w$ has a subsequence in the same relative order as $p$, that is, if there exist $i_1 < i_2 < \dots < i_k$ such that $w_{i_a} < w_{i_b}$ if and only if $p_{i_a} < p_{i_b}$. We say that $w$ \emph{avoids} $p$ if $w$ does not contain the pattern $p$. We will sometimes abuse terminology and refer to either $p \in S_k$ or $w_{i_1}\cdots w_{i_k}$ as being a pattern of $w$.

A permutation $w \in S_n$ is called \emph{dominant} if it avoids the pattern $132$. Equivalently, a permutation is dominant if and only if its code is nonincreasing, that is, $c_1 \geq c_2 \geq \dots \geq c_n$.

\subsubsection{Direct and skew sum}
The following two operations can be used to combine permutations.

\begin{Def} The \emph{direct sum} of permutations $u \in S_m$ and $v \in S_n$ is the permutation $u \oplus v \in S_{m+n}$ defined by
	\[(u \oplus v)(i) = \begin{cases} u(i) &\text{ if } i \leq m, \\ 
		v(i-m)+m &\text{ if } i > m.\end{cases}\]
The \emph{skew sum} of $u \in S_m$ and $v \in S_n$ is the permutation $u \ominus v \in S_{m+n}$ defined by
	\[(u \ominus v)(i) = \begin{cases} u(i)+n &\text{ if } i \leq m, \\ 
		v(i-m) &\text{ if } i > m. \end{cases}\]
\end{Def}

\begin{Def}
	A permutation is called \emph{separable} if it can be built from copies of the permutation $1 \in S_1$ using only direct sum and skew sum operations. 
\end{Def}
In \cite{BBL}, it was shown that separable permutations can alternatively be described as those that avoid the patterns $2413$ and $3142$.

\subsection{Schubert polynomials}

The symmetric group $S_n$ acts on $\ZZ[x_1,...,x_n]$ in a natural way by permuting variables. For instance, if $f \in \ZZ[x_1, \dots, x_n]$, then $s_if$ is the polynomial obtained by switching $x_i$ and $x_{i+1}$ in $f$.

For $i = 1, \dots, n-1$, the \emph{divided difference operator} $\del_i$ is defined by 
\[\del_{i}f=\frac{1-s_{i}}{x_i-x_{i+1}}f = \frac{f-s_if}{x_i-x_{i+1}}\]
for all $f \in \ZZ[x_1, \dots, x_n]$. If $w = s_{i_1} \cdots s_{i_\ell}$ is a reduced expression, then we define $\del_w = \del_{i_1} \cdots \del_{i_\ell}$ (which is independent of the reduced expression).

%
%

The \emph{Schubert polynomials} $\mathfrak{S}_w$ for $w \in S_n$ can be defined recursively as follows: for the long word $w_0\in S_n$, $\mathfrak S_{w_0} = x_1^{n-1}x_2^{n-2} \cdots x_{n-1}$. Otherwise, 
\[\mathfrak S_{ws_i} = \del_i(\mathfrak S_w) \qquad \text{if $\ell(ws_i)<\ell(w)$}.\]
Equivalently, $\mathfrak S_w = \del_{w^{-1}w_0}(x_1^{n-1}x_2^{n-2} \cdots x_{n-1})$ for all $w \in S_n$.

Schubert polynomials are stable under the natural embeddings $\iota \colon S_n \hookrightarrow S_{n+1}$, which implies that $\mathfrak S_w$ is well-defined for any $w \in S_\infty$. The set $\{\mathfrak S_w \mid w \in S_\infty\}$ forms a basis for the polynomial ring $\ZZ[x_1, x_2, \dots]$ called the \emph{Schubert basis}.

The expansion of any Schubert polynomial in terms of monomials has nonnegative coefficients. One combinatorial interpretation for these coefficients is as follows (see \cite{BB, BJS, FK} for more details).

A \emph{pipe dream} (or \emph{rc-graph}) is a type of wiring diagram in which each box $(i,j)$ with $i,j \geq 1$ (indexed using matrix conventions) contains either a cross or a pair of elbows. (See Figure~\ref{fig:4132}.) A pipe dream corresponds to the permutation $w \in S_\infty$ if the wire that enters at the left of row $i$ exits at the top of column $w_i$. A pipe dream is called \emph{reduced} if no two wires cross more than once. 

Every reduced pipe dream for $w$ contains exactly $\ell(w)$ crosses. Assign to each cross the weight $x_i$ if it occurs in row $i$, and define the weight of the pipe dream to be the product of the weights of its crosses. Then $\mathfrak S_w$ is the sum of the weights of all reduced pipe dreams for $w$.

\begin{Ex}\label{ex:4132} Let $w=4132$. Figure~\ref{fig:4132} shows the two reduced pipe dream corresponding to $w$. Hence $\mathfrak S_{4132} = x_1^3x_2+x_1^3x_3$.
\end{Ex}

\begin{figure}
\begin{tikzpicture}[color = black, line width = 1.3pt, scale=.8]

\draw[step = 1, gray, very thin] (-1.9,-2.9) grid (2.9,1.9);

\draw (-.5,1.5) node (b1) {$1$};
\draw (.5,1.5) node (b2) {$2$};
\draw (1.5,1.5) node (b3) {$3$};
\draw (2.5,1.5) node (b4) {$4$};
\draw (-1.5,.5) node (a1) {$4$};
\draw (-1.5,-.5) node (a2) {$1$};
\draw (-1.5,-1.5) node (a3) {$3$};
\draw (-1.5,-2.5) node (a3) {$2$};

\draw (-.5,1) -- (-.5,0);
\draw (-1,.5) -- (0,.5);

\draw (.5,1) -- (.5,0);
\draw (0,.5) -- (1,.5);

\draw (1.5,1) -- (1.5,0);
\draw (1,.5) -- (2,.5);

\draw (-1,-.5) arc (-90:0:.5);
\draw (-.5,-1) arc (180:90:.5);

\draw (-1,-1.5) -- (0,-1.5);
\draw (-.5,-1) -- (-.5,-2);

\draw (0,-.5) arc (-90:0:.5);
\draw (.5,-1) arc (180:90:.5);

\draw (1,-.5) arc (-90:0:.5);

\draw (-1,-2.5) arc (-90:0:.5);

\draw (0,-1.5) arc (-90:0:.5);

\draw (2,.5) arc (-90:0:.5);


\begin{scope}[shift={(-6,0)}]
	\draw[step = 1, gray, very thin] (-1.9,-2.9) grid (2.9,1.9);
	
	\draw (-.5,1.5) node (b1) {$1$};
	\draw (.5,1.5) node (b2) {$2$};
	\draw (1.5,1.5) node (b3) {$3$};
	\draw (2.5,1.5) node (b4) {$4$};
	\draw (-1.5,.5) node (a1) {$4$};
	\draw (-1.5,-.5) node (a2) {$1$};
	\draw (-1.5,-1.5) node (a3) {$3$};
	\draw (-1.5,-2.5) node (a3) {$2$};
	
	\draw (-.5,1) -- (-.5,0);
	\draw (-1,.5) -- (0,.5);
	
	\draw (.5,1) -- (.5,0);
	\draw (0,.5) -- (1,.5);
	
	\draw (1.5,1) -- (1.5,0);
	\draw (1,.5) -- (2,.5);
	
	\draw (-1,-.5) arc (-90:0:.5);
	\draw (-.5,-1) arc (180:90:.5);
	
	\draw (-1,-1.5) arc (-90:0:.5);
	\draw (-.5,-2) arc (180:90:.5);
	
	\draw (0,-.5) -- (1,-.5);
	\draw (.5,0) -- (.5,-1);
	
	\draw (1,-.5) arc (-90:0:.5);

	\draw (-1,-2.5) arc (-90:0:.5);
	
	\draw (0,-1.5) arc (-90:0:.5);
	
	\draw (2,.5) arc (-90:0:.5);
	
\end{scope}
\end{tikzpicture}
\caption{The two reduced pipe dreams for $4132$. (In this diagram, only the first four pipes are drawn; all other pipes consist only of elbows.)}
\label{fig:4132}
\end{figure}
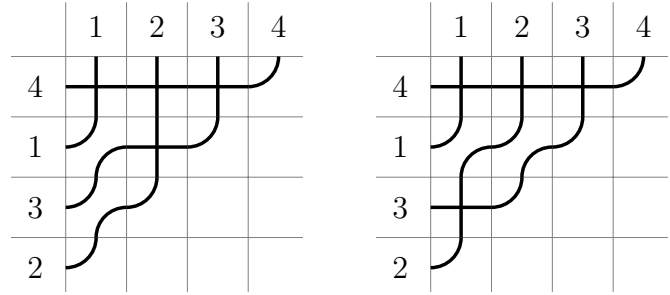

One special case of Schubert polynomials occurs when $w$ is a dominant (132-avoiding) permutation. In this case, $\mathfrak S_w$ is the monomial $x_1^{c_1}x_2^{c_2} \cdots$, where $(c_1, c_2, \dots)$ is the code of $w$.

Another special case occurs when $w$ is a \emph{Grassmannian permutation} satisfying $w_1 < w_2 < \cdots < w_r$ and $w_{r+1} < w_{r+2} < \cdots < w_n$ for some $r$. In this case, $\mathfrak S_w$ is a symmetric polynomial in $x_1, \dots, x_r$ called a \emph{Schur polynomial} $s_\lambda(x_1, \dots, x_r)$, where $\lambda$ is the partition $(w_r-r, w_{r-1}-(r-1), \dots, w_1-1)$.

The following proposition describes how Schubert polynomials behave under direct sum and skew sum.

\begin{Prop} \label{prop:schubertsum}
		Let $u \in S_m$ and $v \in S_n$. Then:
		\begin{enumerate}[(a)]
			\item $\mathfrak{S}_{u \oplus v} = \mathfrak{S}_{u} \cdot \mathfrak{S}_{1_m \oplus v}$, and
			\item $\mathfrak{S}_{u \ominus v} = \mathfrak S_u \cdot \mathfrak S_{1_m \ominus v} = \mathfrak S_u \cdot (x_1\cdots x_m)^n \cdot \mathfrak S_v(x_{m+1}, \dots, x_{m+n})$.
		\end{enumerate}
\end{Prop}
\begin{proof}
	For (a), any reduced pipe dream for $u \oplus v$ must have the first $m$ pipes lying strictly above the last $n$ pipes. Thus such a pipe dream can be factored uniquely into a reduced pipe dream for $u$ and (by replacing the first $m$ pipes with the identity pipe dream containing only elbows) a reduced pipe dream for $1_m \oplus v$.
	
	For (b), any reduced pipe dream for $u \ominus v$ must have crosses in the first $n$ boxes of the first $m$ rows. The remaining part consists of a reduced pipe dream for $u$ (shifted to the right by $n$) and a reduced pipe dream for $v$ (shifted down by $m$). The result follows easily.
\end{proof}

\subsection{Standard elementary monomials} \label{sec:sem}

For integers $j$ and $k$ with $k \geq 0$, denote by 
\[e_j^{(k)} = \sum_{1 \leq i_1 < \dots < i_j \leq k} x_{i_1} \cdots x_{i_j}\]
the $j$th \emph{elementary symmetric polynomial} in $x_1, \dots, x_k$. (By convention, $e_j^{(k)} = 1$ for $j=0$, while $e_j^{(k)} = 0$ if $j > k$ or $j < 0$.) Note that $e_j^{(k)}$ is symmetric in $x_i$ and $x_{i+1}$ for all $i \neq k$.

Let $L$ be the set of sequences of integers $(j_1, j_2, \dots)$ satisfying $0 \leq j_k \leq k$ for which all but finitely many of the $j_k$ vanish. (We will sometimes omit trailing zeroes from such sequences for convenience.) Then for any $(j_1, j_2, \dots) \in L$ we define the \emph{standard elementary monomial} $e_{j_1j_2\cdots}$ to be the polynomial
\[e_{j_1j_2 \cdots} = \prod_{k \geq 1} e_{j_k}^{(k)}.\] (Note that all but finitely many terms in the product are $1$.)

It was shown in \cite{FGP} that as $(j_1, j_2, \dots)$ ranges over all sequences in $L$, the standard elementary monomials $e_{j_1 j_2\cdots}$ form a basis for the polynomial ring $\ZZ[x_1, x_2, \dots]$, which we call the \emph{SEM basis}. 
(Though we will not need it here, each standard elementary monomial has nonnegative coefficients when expanded in the Schubert basis, as determined by the Pieri rule for Schubert polynomials---see, for instance, \cite{LS}.)

Given a permutation $w \in S_n$, consider the expansion of the corresponding Schubert polynomial in the SEM basis
\[\mathfrak S_w = \sum \alpha_{j_1j_2\cdots j_{n-1}} e_{j_1j_2\cdots j_{n-1}}.\]
Most notably, this expansion appears in the study of quantum Schubert calculus: Fomin, Gelfand, and Postnikov \cite{FGP} define the \emph{quantum Schubert polynomial} as
\[\mathfrak S^q_w = \sum \alpha_{j_1j_2 \cdots j_{n-1}} E_{j_1j_2 \cdots j_{n-1}},\]
where $E_{j_1j_2 \cdots j_{n-1}} = \prod_k E_{j_k}^{(k)}$ is a product of \emph{quantum elementary polynomials} $E_{j}^{(k)} = E_{j}(x_1, \dots, x_k)$ defined by
\[\det(I + \lambda G_k) = \sum_{j=0}^k E^{(k)}_j \lambda^j, \quad \text{where} \quad G_k = \begin{bmatrix}
	x_1&q_1&0&\cdots&0\\
	-1&x_2&q_2&\cdots&0\\
	0&-1&x_3&\cdots&0\\
	\vdots&\vdots&\vdots&\ddots&\vdots\\
	0&0&0&\cdots&x_k
\end{bmatrix}.\]
Hence any formula for the SEM expansion of Schubert polynomials may also be thought of as a formula for quantum Schubert polynomials. See \cite{FGP} for further background on quantum Schubert polynomials and their role in the quantum cohomology of the flag variety.

In \cite[Section 17]{PS}, it is shown that the coefficients $\alpha_{j_1j_2\cdots}$ are also the entries in the \emph{inverse Schubert-Kostka matrix} expressing monomials in terms of the Schubert basis. In general, computational evidence suggests that most of the $\alpha_{j_1j_2\cdots}$ are small in absolute value. For instance, all such coefficients have absolute value at most $1$ when $n \leq 6$---see Winkel \cite{Winkel} for more observation and discussion about these coefficients.

\subsection{Nonintersecting lattice paths}

A key result for finding determinantal formulas is the following Lindstr\"om-Gessel-Viennot lemma \cite{GV, Lindstrom}.

Let $G = (V,E)$ be a locally finite acyclic directed graph, and suppose that each edge $e \in E$ is assigned an edge weight $w_e$ (lying in some commutative ring). For any path in $G$, we define its weight to be the product of the weights of all edges in the path. For any two vertices $a$ and $b$, we will write $e(a,b)$ for the total weight of all directed paths from $a$ to $b$.

Let $A = \{a_1, \dots, a_k\}$ and $B = \{b_1, \dots, b_k\}$ be subsets of $V$ of the same size. A collection of \emph{nonintersecting paths} $P=(P_1, \dots, P_k)$ from $A$ to $B$ is a sequence of vertex-disjoint paths such that, for some permutation $\sigma \in S_k$, $P_i$ is a directed path from $a_i$ to $b_{\sigma(i)}$ for all $i$. Denote by $\paths(A,B)$ the set of all such $P$. For any $P \in \paths(A,B)$, we will write $\sigma(P)$ for the corresponding permutation $\sigma$ and $w(P)$ for the product of the weights of paths in $P$.

\begin{Lemma}[Lindstr\"om-Gessel-Viennot] \label{LGV}
	Let $G$, $A$, and $B$ be defined as above. Then the signed sum of the weights of all collections of nonintersecting paths from $A$ to $B$ is given by the determinant
	\[\sum_{P \in \paths(A,B)} \operatorname{sgn}(\sigma(P)) \cdot w(P) = \det (e(a_i, b_j))_{i,j=1}^k.\]
\end{Lemma}

In particular, if $\sigma(P)$ is the identity permutation for all $P$, then the left hand side is just the sum of the weights of all collections of nonintersecting paths.

One standard application of Lemma~\ref{LGV} is the (dual) \emph{Jacobi-Trudi identity}. 

\begin{Prop}[Dual Jacobi-Trudi] \label{prop:jacobitrudi}
	Let $\lambda$ be a partition with largest part $r$. Then the Schur polynomial $s_\lambda(x_1, \dots, x_n)$ is given by the determinant
	\[s_\lambda(x_1, \dots, x_n) = \det(e^{(n)}_{\lambda_i'+j-i})_{i,j=1}^r,\]
	where each entry is an elementary symmetric polynomial in $x_1, \dots, x_n$.
\end{Prop}

(Here, $\lambda'$ is the conjugate partition to $\lambda$, so that for any positive integer $i$, $\lambda'_i = \#\{j \mid \lambda_j \geq i\}$.)

The proof of this result involves applying Lemma~\ref{LGV} on the following graph. Let $G$ have vertex set $\ZZ \times \ZZ_{\geq 0}$---by convention, we will draw the positive $x$-axis in the rightward direction and the positive $y$-axis in the vertical direction. Whenever both endpoints lie in $G$, add a directed edge from $(a,b)$ to $(a,b+1)$ of weight $x_{b+1}$ (which we call an ``upstep''), as well as a directed edge from $(a,b)$ to $(a-1,b+1)$ of weight $1$ (which we call a ``diagonal step''). See Figure~\ref{fig:graph}.

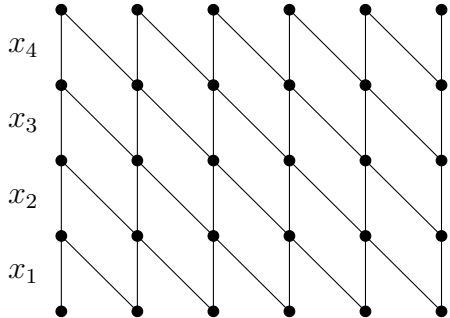
\begin{figure}
	\begin{tikzpicture}
		\foreach \x in {0,...,5}
			\foreach \y in {0,...,4}
				\node[v,fill=black] at (\x,\y){};
		\foreach \x in {0,...,5}
			\foreach \y in {0,...,3}
				\draw (\x,\y)--(\x,\y+1);
		\foreach \x in {1,...,5}
			\foreach \y in {0,...,3}
				\draw (\x,\y)--(\x-1,\y+1);
		\foreach \y in {1,...,4}
			\node at (-.5,\y-.5) {$x_\y$};
	\end{tikzpicture}
	\caption{A portion of the graph $G$ on $\ZZ \times \ZZ_{\geq 0}$. All vertical edges are directed up, weighted according to height as shown, while all diagonal edges are directed up with weight $1$.}
	\label{fig:graph}
\end{figure}

Observe that any directed path from $(a,0)$ to $(b,c)$ must use $a-b$ diagonal steps and $c+b-a$ upsteps. Moreover, each upstep must occur at a different one of the $c$ possible heights. It follows that $e((a,0),(b,c)) = e^{(c)}_{c+b-a}$. Applying Lemma~\ref{LGV} then immediately implies the following result.

\begin{Prop} \label{edet}
	Let $G$ be defined as above, and let 
	\begin{align*}
		A &= \{(a_1,0), (a_2,0), \dots, (a_k,0)\},\\
		B &= \{(b_1, c_1), (b_2, c_2), \dots (b_k, c_k)\}.
	\end{align*}
	Then
	\[\sum_{P \in \paths(A,B)} \operatorname{sgn}(\sigma(P)) \cdot  w(P) = \det (e^{(c_j)}_{c_j+b_j-a_i})_{i,j=1}^k,\]
	where $\paths(A,B)$ is the set of all collections of nonintersecting paths from $A$ to $B$.
\end{Prop}

\begin{Def}
	A polynomial $F$ has a \emph{lattice path representation} $(A,B)$ if
		\begin{align*}
		A &= \{(a_1,0), (a_2,0), \dots, (a_k,0)\},\\
		B &= \{(b_1, c_1), (b_2, c_2), \dots (b_k, c_k)\},
	\end{align*}
	and
	$F = \sum_{P \in \paths(A,B)} \operatorname{sgn}(\sigma(P)) \cdot w(P)$, where $P$ ranges over all collections of nonintersecting paths from $A$ to $B$ as in Proposition~\ref{edet}.
	
	The multiset of \emph{heights} of $(A,B)$ is $\{c_1, \dots, c_k\}$.
\end{Def}

\begin{Ex}
	Consider the Schubert polynomial $\mathfrak S_{4132}$ as in Example~\ref{ex:4132}. One can verify that
	\[
		\mathfrak S_{4132} = x_1^3x_2+x_1^3x_3
		=e_{112}-e_{103}-e_{022}
		= \left|\begin{matrix}
		e_1^{(1)}&e_2^{(2)}&0\\
		e_0^{(1)}&e_1^{(2)}&e_3^{(3)}\\
		0&e_0^{(2)}&e_2^{(3)}
		\end{matrix}\right|.
	\]
	This corresponds to the lattice path representation $(A,B)$ with
	\[A = \{(0,0), (1,0), (2,0)\}, \qquad B=\{(0,1), (0,2), (1,3)\}\]
	whose nonintersecting paths are depicted in Figure~\ref{fig:4132-2}.
\end{Ex}

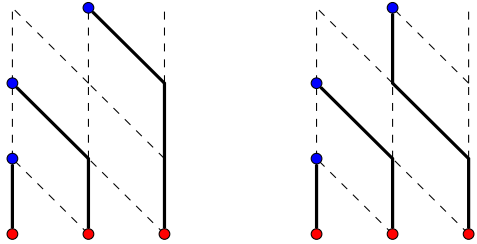
\begin{figure}
	\begin{tikzpicture}
		\foreach \x in {0,...,2}
		\foreach \y in {0,...,2}
		\draw[dashed] (\x,\y)--(\x,\y+1);
		\foreach \x in {1,...,2}
		\foreach \y in {0,...,2}
		\draw[dashed] (\x,\y)--(\x-1,\y+1);
		\node[v] (a1) at (0,0){};
		\node[v] (a2) at (1,0){};
		\node[v] (a3) at (2,0){};
		\node[w] (b1) at (0,1){};
		\node[w] (b2) at (0,2){};
		\node[w] (b3) at (1,3){};
		\draw[very thick] (a1)--(b1);
		\draw[very thick] (a2)--(1,1)--(b2);
		\draw[very thick] (a3)--(2,1)--(1,2)--(b3);
		
		\begin{scope}[shift={(-4,0)}]
			\foreach \x in {0,...,2}
			\foreach \y in {0,...,2}
			\draw[dashed] (\x,\y)--(\x,\y+1);
			\foreach \x in {1,...,2}
			\foreach \y in {0,...,2}
			\draw[dashed] (\x,\y)--(\x-1,\y+1);
			\node[v] (a1) at (0,0){};
			\node[v] (a2) at (1,0){};
			\node[v] (a3) at (2,0){};
			\node[w] (b1) at (0,1){};
			\node[w] (b2) at (0,2){};
			\node[w] (b3) at (1,3){};
			\draw[very thick] (a1)--(b1);
			\draw[very thick] (a2)--(1,1)--(b2);
			\draw[very thick] (a3)--(2,1)--(2,2)--(b3);
		\end{scope}
	\end{tikzpicture}
	\caption{Nonintersecting lattice paths for a lattice path representation of $\mathfrak S_{4132} = x_1^3x_2+x_1^3x_3$. (The points in $A$ and $B$ are given by red and blue nodes, respectively.)}
	\label{fig:4132-2}
\end{figure}

The order of the labelings of the points in $A$ and $B$ only affects $F$ up to a sign. Therefore we will often abuse notation slightly by considering $A$ and $B$ as unordered sets for ease of exposition. In most of the situations that we will consider, each collection $P$ of nonintersecting paths will have the same $\sigma(P)$, and so we can label the elements of $B$ so that $\sigma(P)$ is the identity.

The lattice path representation of a polynomial is not unique: for example, the constant polynomial $1$ can be represented by any pair $(A,B)$ such that $a_i = b_i+c_i$ for all $i$ (as the corresponding matrix will be upper triangular with $1$'s on the diagonal).

Given a determinantal expression whose entries are elementary symmetric polynomials that vary in the appropriate way, it is straightforward to find corresponding sets $A$ and $B$.

\begin{Ex}
	The dual Jacobi-Trudi identity involves a determinant whose $(i,j)$th entry is given by $e^{(n)}_{\lambda_i'+j-i}$. This can be obtained from Proposition~\ref{edet} by setting, for instance, $a_i = n+i-\lambda_i'$, $b_j = j$, and $c_j=n$.
	
	One can then give a weight-preserving bijection between $\paths(A,B)$ and, for instance, semistandard Young tableaux of shape $\lambda$ to deduce the dual Jacobi-Trudi identity: see \cite{GV}.
\end{Ex}

Observe that if the heights of a lattice path representation are distinct, then in the expansion of the determinant in Proposition~\ref{edet}, each term either vanishes or equals, up to sign, a standard elementary monomial. In addition, all of the nonzero terms obtained in this way will necessarily be distinct. Therefore when this occurs, this determinant can be thought of as a concise representation of the SEM expansion of the resulting polynomial. We will call such representations \emph{proper}.

\begin{Def}
	Let $\{c_1, \dots, c_k\}$ be the multiset of heights of a lattice path representation $(A,B)$. We say that $(A,B)$ is \emph{proper} if the $c_i$ are distinct.
\end{Def}

Our goal for most of the remainder of this paper is to investigate which Schubert polynomials have a proper lattice path representation.

\subsection{Quantization}

As a brief digression, we will first discuss a slight modification of these lattice path representations for computing quantum Schubert polynomials. (This section will not be needed for the remainder of this paper.)

As described in \S\ref{sec:sem}, the quantum Schubert polynomials $\mathfrak S_w^q$ are defined by computing the SEM expansion of $\mathfrak S_w$ and replacing each elementary polynomial $e_j^{(k)}$ with the quantum elementary polynomial $E_j^{(k)}$. In the event that $\mathfrak S_w$ has a proper lattice path representation and hence a determinantal formula for its SEM expansion by Proposition~\ref{edet}, it follows that $\mathfrak S_w^q$ is also expressible as a determinant whose entries are of the form $E_j^{(k)}$. In fact, there exists a simple modification to our underlying graph $G$ on $\ZZ \times \ZZ_{\geq 0}$ that yields the quantum elementary polynomials as weights.

Let $G^q$ be the graph on $\ZZ \times \ZZ_{\geq 0}$ with the same edges as $G$ as before but with additional edges from $(a,b)$ to $(a,b+2)$ of weight $q_{b+1}$. (Thus if we set all $q_i=0$, then the graph $G^q$ essentially reverts to the original graph $G$.)

\begin{Prop} \label{prop:gq}
	The total weight $e((a,0),(b,c))$ of all paths from $(a,0)$ to $(b,c)$ in $G^q$ is $E_{c+b-a}^{(c)}$.
\end{Prop}
\begin{proof}
	Recall that $E_j^{(k)}$ is defined to be the coefficient of $\lambda^j$ in the expansion of $\det(I+\lambda G_k)$. Expanding the determinant along the last column of $I+\lambda G_k$ gives
	\[E_j^{(k)}=E_j^{(k-1)}+x_{k} E_{j-1}^{(k-1)} + q_{k-1} E_{j-2}^{(k-2)}.\]
	
	Similarly, any path in $G^q$ from $(a,0)$ ending at $(b,c)$ must come from $(b+1,c-1)$ with an edge of weight $1$, from $(b,c-1)$ with an edge of weight $x_{c}$, or from $(b,c-2)$ with an edge of weight $q_{c-1}$. Hence $e((a,0),(b,c))$ equals
	\[e((a,0),(b+1,c-1)) + x_{c}e((a,0),(b,c-1)) + q_{c-1}e((a,0),(b,c-2)).\]
	Since $e((a,0),(b,c))$ and $E_{c+b-a}^{(c)}$ satisfy the same base cases (equaling $1$ if $c+b-a=0$ and $0$ if $c+b-a<0$), the result follows easily by induction.
\end{proof}

The following corollary is then immediate.

\begin{Cor}
	Suppose $\mathfrak S_w$ has a proper lattice path representation $(A,B)$. Then
	\[\mathfrak S_w^q = \sum_{P \in \mathscr P^q(A,B)} \operatorname{sgn}(\sigma(P)) \cdot w(P),\]
	where $\mathscr P^q(A,B)$ is the set of all collections of nonintersecting paths from $A$ to $B$ in the graph $G^q$.
\end{Cor}
\begin{proof}
	Combine Proposition~\ref{edet} and the definition of $\mathfrak S_w^q$ with Proposition~\ref{prop:gq} using Lemma~\ref{LGV}.
\end{proof}

As we will see, a large class of permutations $w$ to which this corollary applies will be described by our main result, Theorem~\ref{thm:main}.

\section{Operations} 

In this section, we will describe several operations on lattice path representations that act predictably on the corresponding polynomials.

\begin{Prop} \label{prop:pull}
	Let $(A,B)$ be a lattice path representation of a polynomial $F$,
		and suppose $(b,c), (b+1,c) \in B$. Then $(A,B')$ is a lattice path representation for $F$, where $B'$ is formed by replacing $(b+1,c)$ with $(b,c+1)$ in $B$.
\end{Prop}

\begin{proof}
	Any path that ends at $(b,c+1)$ that does not pass through $(b,c)$ must end with a diagonal step from $(b+1,c)$. Removing this last diagonal step (which has weight $1$) then gives a weight-preserving bijection from $\paths(A,B')$ to $\paths(A,B)$.
\end{proof}

One can alternatively prove Proposition~\ref{prop:pull} from the determinantal expression $F = \det E$ given in Proposition~\ref{edet}. Note that $E$ contains two columns whose entries in each row $i$ have the form $e^{(c)}_{c+b-a_i}$ and $e^{(c)}_{c+b+1-a_i}$. Adding $x_{c+1}$ times the first column to the second does not change the value of the determinant. The entries in the second column then become
\[x_{c+1}e^{(c)}_{c+b-a_i} + e^{(c)}_{c+b+1-a_i} = e^{(c+1)}_{c+b+1-a_i},\]
so that the resulting matrix corresponds to the new representation $(A,B')$.
%
%
%

Our next operation concerns the action of the divided difference operators $\partial_i$. For a similar result, see \cite[Lemma 4.4]{CLL}.

\begin{Prop} \label{prop:drop}
	Let $(A,B)$ be a lattice path representation of a polynomial $F$, and suppose that $B$ has a unique point $(b,c)$ at height $c$. Then $(A,B')$ is a lattice path representation for $\del_c(F)$, where $B'$ is formed by replacing $(b,c)$ with $(b,c-1)$ in $B$.
	
	If instead $B$ has no point at height $c$, then $\del_c(F)=0$.
\end{Prop}

\begin{proof}
	By Proposition~\ref{edet}, $F$ is the determinant of a matrix $(e^{(c_j)}_{c_j+b_j-a_i})_{i,j=1}^k$. Each entry of this matrix is symmetric in $x_c$ and $x_{c+1}$ unless $c = c_j$, which occurs in a unique column (since $B$ has a unique point at height $c$). Then in the Laplace expansion of the determinant along this column, each term has the form $e^{(c)}_{c+b-a_i} \cdot g$ for some minor $g$ that is symmetric in $x_c$ and $x_{c+1}$. Applying $\partial_c$ then gives
	\[\partial_c(e^{(c)}_{c+b-a_i} \cdot g) = \partial_c(e^{(c)}_{c+b-a_i}) \cdot g = e^{(c-1)}_{(c-1)+b-a_i} \cdot g.\]
	Thus $\partial_c$ has the effect of replacing $e^{(c)}_{c+b-a_i}$ with $e^{(c-1)}_{(c-1)+b-a_i}$ in the determinant. By Proposition~\ref{edet}, this new determinant for $\partial_c(F)$ corresponds to the lattice path representation $(A,B')$, as desired.
	
	If instead $B$ has no point at height $c$, then every entry of the determinant for $F$ is symmetric in $x_c$ and $x_{c+1}$, so $\partial_c(F)=0$.
\end{proof}

By combining Propositions~\ref{prop:pull} and \ref{prop:drop}, we arrive at the following operation that preserves heights.

\begin{Prop}\label{prop:slide}
	Let $(A,B)$ be a lattice path representation for $F$, and suppose that $B$ has a unique point $(b,c)$ at height $c$.
	\begin{enumerate}[(a)]
		\item If $(b+1,c-1) \in B$, then $(A,B')$ is a lattice path representation for $-\del_c(F)$, where $B'$ is formed by replacing $(b+1,c-1)$ by $(b,c-1)$ in $B$.
		\item If $(b-1,c-1) \in B$, then $(A,B'')$ is a lattice path representation for $\del_c(F)$, where $B''$ is formed by replacing $(b,c)$ by $(b-1,c)$ in $B$.
	\end{enumerate}
\end{Prop}
\begin{proof}
	Apply Proposition~\ref{prop:drop} to the point $(b,c)$, and then apply Proposition~\ref{prop:pull} to the two points at height $c-1$.
\end{proof}

\begin{Ex}
	Let $F=x_1^3x_2^2+x_1^3x_2x_3$, which has lattice path representation
	\[A = \{(0,0), (1,0), (2,0)\}, \qquad B= \{(0,2), (1,1), (1,3)\}\]
	as shown in the middle diagram of Figure~\ref{fig:slide}.
	
	Applying Proposition~\ref{prop:slide}(a) with $c=2$ shows that we can obtain a representation for $\del_2 F = x_1^3x_2+x_1^3x_3$ by moving the endpoint $(1,1)$ to $(0,1)$ (and permuting the set $B$ appropriately to get rid of the sign), as shown on the left of Figure~\ref{fig:slide}.
	
	Alternatively, applying Proposition~\ref{prop:slide}(b) with $c=3$ shows that we can obtain a representation for $\del_3 F = x_1^3x_2$ by moving the endpoint $(1,3)$ to $(0,3)$, as shown on the right of Figure~\ref{fig:slide}.
\end{Ex}

\begin{figure}
		\begin{tikzpicture}
		\foreach \x in {0,...,2}
		\foreach \y in {0,...,2}
		\draw[dashed] (\x,\y)--(\x,\y+1);
		\foreach \x in {1,...,2}
		\foreach \y in {0,...,2}
		\draw[dashed] (\x,\y)--(\x-1,\y+1);
		\node[v] (a1) at (0,0){};
		\node[v] (a2) at (1,0){};
		\node[v] (a3) at (2,0){};
		\node[w] (b1) at (0,2){};
		\node[w] (b2) at (1,1){};
		\node[w] (b3) at (1,3){};
		\draw[very thick] (a1)--(0,1)--(b1);
		\draw[very thick] (a2)--(b2);
		\draw[very thick] (a3)--(2,1)--(1,2)--(b3) (2,1)--(2,2)--(b3);
		\draw[->] (-1,1.5)-- node[above]{$\del_2$} (-2,1.5);
		\draw[->] (3,1.5)--node[above]{$\del_3$} (4,1.5);
		
		\begin{scope}[shift={(-5,0)}]
			\foreach \x in {0,...,2}
			\foreach \y in {0,...,2}
			\draw[dashed] (\x,\y)--(\x,\y+1);
			\foreach \x in {1,...,2}
			\foreach \y in {0,...,2}
			\draw[dashed] (\x,\y)--(\x-1,\y+1);
			\node[v] (a1) at (0,0){};
			\node[v] (a2) at (1,0){};
			\node[v] (a3) at (2,0){};
			\node[w] (b1) at (0,1){};
			\node[w] (b2) at (0,2){};
			\node[w] (b3) at (1,3){};
			\draw[very thick] (a1)--(b1);
			\draw[very thick] (a2)--(1,1)--(b2);
			\draw[very thick] (a3)--(2,1)--(1,2)--(b3) (2,1)--(2,2)--(b3);
		\end{scope}
	
		\begin{scope}[shift={(5,0)}]
			\foreach \x in {0,...,2}
			\foreach \y in {0,...,2}
			\draw[dashed] (\x,\y)--(\x,\y+1);
			\foreach \x in {1,...,2}
			\foreach \y in {0,...,2}
			\draw[dashed] (\x,\y)--(\x-1,\y+1);
			\node[v] (a1) at (0,0){};
			\node[v] (a2) at (1,0){};
			\node[v] (a3) at (2,0){};
			\node[w] (b1) at (0,2){};
			\node[w] (b2) at (1,1){};
			\node[w] (b3) at (0,3){};
			\draw[very thick] (a1)--(0,1)--(b1);
			\draw[very thick] (a2)--(b2);
			\draw[very thick] (a3)--(2,1)--(1,2)--(b3);
		\end{scope}
		\end{tikzpicture}
	\caption{Application of Proposition~\ref{prop:slide}. The center picture gives a lattice path representation for $F=x_1^3x_2^2+x_1^3x_2x_3$. (Both sets of nonintersecting paths are overlaid for conciseness.) The left and right pictures represent $\del_2 F = x_1^3x_2+x_1^3x_3$ and $\del_3F = x_1^3x_2$, respectively.}
	\label{fig:slide}
\end{figure}
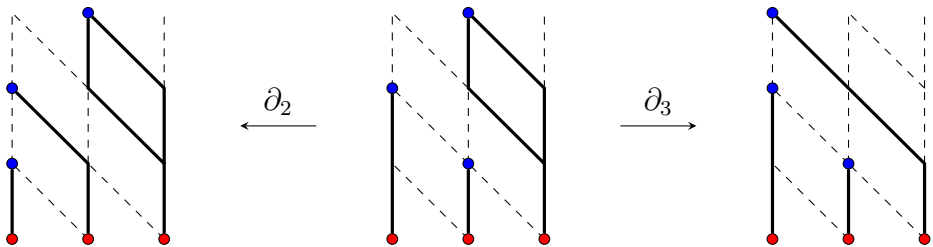

Our last operation concerns products of polynomials. Observe that there exists a directed path from $(a,0)$ to $(b,c)$ if and only if $b \leq a \leq b+c$.

\begin{Prop} \label{prop:product}
	Let $(A,B)$ and $(A',B')$ be lattice path representations for polynomials $F$ and $G$, respectively, such that there do not exist any directed paths from a point in $A$ to a point in $B'$. Then $(A \cup A', B \cup B')$ is a lattice path representation for the product $FG$.
\end{Prop}
\begin{proof}
	By the given condition, the only points of $A \cup A'$ that points in $B'$ can be connected to are those in $A'$. No paths from $A$ to $B$ intersect any paths from $A'$ to $B'$ (or else there would be a path from $A$ to $B'$), so the elements of $\paths(A \cup A', B \cup B')$ are formed by pairing an element of $\paths(A,B)$ with an element of $\paths(A',B')$.
\end{proof}

Note that one can always translate $(A,B)$ horizontally so that the desired condition holds.

As a special case, we can derive the following result that allows us to delete (or add) certain points from a lattice path representation.

\begin{Prop} \label{prop:delete}
	Let $(A,B)$ be a lattice path representation of a polynomial $F$, and suppose that for some $s \geq 0$,
	\begin{align*}
		A'&=\{(a,0), (a+1,0), \dots, (a+s,0)\} \subseteq A,\\
		B'&=\{(a,0), (a,1), \dots, (a,s)\} \subseteq B.
	\end{align*}
	Then $(A \setminus A', B \setminus B')$  is also a lattice path representation of $F$.
\end{Prop}
\begin{proof}
	There are no directed paths from any point in $A \setminus A'$ to any point in $B'$. Since there is a unique collection of nonintersecting paths from $A'$ to $B'$, and these paths use only diagonal steps, $(A',B')$ is a lattice path representation of $1$. The result then follows from Proposition~\ref{prop:product}.
\end{proof}


\section{Representing Schubert polynomials}

In this section, we will use the operations described in \S3 to construct lattice path representations for a wide range of Schubert polynomials.

\subsection{Compact representations}

We will first investigate a special type of lattice path representation.

\begin{Def}
	A lattice path representation $(A,B)$ is \emph{compact} if $\{a_1, \dots, a_k\} = \{c_1, \dots, c_k\} = \{0, \dots, k-1\}$, and $0 \leq b_i \leq k-1$ for all $i$, where $k = |A|=|B|$.
\end{Def}

In other words, a compact lattice path representation is proper, and all endpoints fit within a square of side length $k-1$, where $k = |A| = |B|$. Note that in order for there to exist at least one set of nonintersecting lattice paths, we must have that at least $s$ of the $b_i$ are less than $s$ (so that the paths starting at the first $s$ points of $A$ have endpoints), that is, the sequence of $b_i$ must be a \emph{parking function}.

Our main result of this section will be the following theorem.

\begin{Th} \label{thm:compact}
	Let $w \in S_n$ be a permutation that avoids $1324$, $2413$, and $3142$. Then $\mathfrak S_w$ has a compact lattice path representation. 
\end{Th}

Recall that a permutation is called separable if it avoids $2413$ and $3142$. Hence the permutations in the theorem above are the $1324$-avoiding separable permutations.

To prove this theorem, we first consider the special cases of dominant ($132$-avoiding) permutations and $213$-avoiding permutations.

\begin{Lemma} \label{lemma:132}
	Let $w \in S_n$ be a $132$-avoiding permutation. Then $(-1)^{\binom{n}{2} - \ell(w)} \mathfrak S_w$ has a compact lattice path representation $(A,B)$, where
	\begin{align*}
		A &= \{(n-1,0), (n-2,0), \dots, (0,0)\},\\
		B &= \{(b_1, 0), (b_2, 1), \dots (b_n, n-1)\},
	\end{align*}
	where $(b_1, b_2, \dots, b_n)$ is the code of $w$.
\end{Lemma}
\begin{proof}
	We will induct on $\binom{n}{2} - \ell(w)$. When $w = w_0$, $b_i = n-i$, and the only way to connect $A$ and $B$ with nonintersecting lattice paths is via vertical paths, which have combined weight $\mathfrak S_{w_0} = x_1^{n-1}x_2^{n-2} \cdots x_{n-1}$.
	
	Suppose $w \neq w_0$ and let $(b_1, \dots, b_n)$ be the code of $w$. Since $w$ is dominant, we must have $n-1 \geq b_1 \geq b_2 \geq \dots \geq b_n \geq 0$. Since $w \neq w_0$, there exists a minimum index $i$ such that $b_{i} = b_{i+1}$, so that $w_i < w_{i+1}$. Then $w'=ws_i$ has length $\ell(w')=\ell(w)+1$ and has code $(b_1, \dots, b_{i-1}, b_i+1, b_{i+1}, \dots b_n)$. Since this code is still weakly decreasing, $w'$ is also dominant and therefore by induction $(-1)^{\binom{n}{2} - \ell(w')}\mathfrak S_{w'}$ has a corresponding lattice path representation $(A, B')$. 
	
	Note that $B'$ contains the two points $(b_i+1, i-1)$ and $(b_{i+1},i)=(b_i,i)$. Since $B$ is obtained from $B'$ by replacing $(b_i+1,i-1)$ with $(b_i,i-1)$, it follows by Proposition~\ref{prop:slide}(a) that $(-1)^{\binom{n}{2}-\ell(w)}\del_i(\mathfrak S_{w'}) = (-1)^{\binom{n}{2}-\ell(w)}\mathfrak S_w$ has lattice path representation $(A,B)$.
\end{proof}

Alternatively, since $w$ is dominant, $\mathfrak S_w$ is a monomial. Hence one can also prove Lemma~\ref{lemma:132} by verifying that there exists a unique collection of nonintersecting lattice paths from $A$ to $B$ of the appropriate weight.

One can similarly prove the following result for $213$-avoiding permutations. (Note that $w$ is $213$-avoiding if and only if $w_0ww_0$ is dominant.)

\begin{Lemma} \label{lemma:213}
	Let $w \in S_n$ be a $213$-avoiding permutation. Then $\mathfrak S_w$ has compact lattice path representation $(A,B)$, where
	\begin{align*}
		A &= \{(n-1,0), (n-2,0), \dots, (0,0)\},\\
		B &= \{(b_1, n-1), (b_2,n-2), \dots, (b_n,0)\},
	\end{align*}
	where $(b_1, b_2, \dots, b_n)$ is the code of $w_0ww_0$.
\end{Lemma}
\begin{proof}
	We induct on $\binom{n}{2}-\ell(w)$. When $w=w_0$, $b_i = n-i$, and there is a unique set of nonintersecting paths from $A$ to $B$ with weight $\mathfrak S_{w_0}$.
	
	Suppose $w \neq w_0$, and let $u = w_0ww_0$. Since $u$ is dominant, we can define $u'=us_i$ such that $u'$ is dominant as in Lemma~\ref{lemma:132}. Then $w' = w_0u'w_0 = w_0uw_0s_{n-i} = ws_{n-i}$ is also $213$-avoiding with $\ell(w') = \ell(w)+1$. Hence by induction $\mathfrak S_{w'}$ has a corresponding lattice path representation $(A,B')$.
	
	Since (as in Lemma~\ref{lemma:132}) $u'$ has code $(b_1, \dots, b_{i-1},b_{i}+1, b_{i+1}, \dots, b_n)$, $B'$ contains the two points $(b_i+1, n-i)$ and $(b_{i+1},n-i-1) = (b_i,n-i-1)$. But $B$ is obtained from $B'$ by replacing $(b_i+1,n-i)$ with $(b_i, n-i)$, so by Proposition~\ref{prop:slide}(b), $(A,B)$ is a lattice path representation for $\del_{n-i} \mathfrak S_{w'} = \mathfrak S_w$.
\end{proof}

Applying Proposition~\ref{edet} to the lattice path representation in Lemma~\ref{lemma:213} yields a determinantal formula that gives the SEM expansion for $\mathfrak S_w$ when $w$ is $213$-avoiding as in Corollary 17.12 of \cite{PS}.

We are now ready to prove that any $1324$-avoiding separable permutation has a compact lattice path representation.

\begin{proof}[Proof of Theorem~\ref{thm:compact}]
	We proceed by induction on $n$. The case $n=1$ is trivial. For $n > 1$, since $w$ avoids $2413$ and $3142$, it is separable. Hence we can either write $w=u \ominus v$ or $w = u \oplus v$ for separable permutations $u \in S_m$ and $v \in S_{n-m}$ that avoid $1324$.
	
	Suppose first that $w = u \ominus v$. By induction, $u$ and $v$ have compact lattice point representations $(A_u, B_u)$ and $(A_v, B_v)$, respectively. Using addition to indicate translation, we claim that if
	\begin{align*}
	A_w &= (A_u + (n-m,0)) \cup A_v = \{(n-1,0), (n-2,0), \dots, (0,0)\},\\
	B_w &= (B_u+(n-m,0)) \cup (B_v+(0,m)),
	\end{align*}
	then $(A_w, B_w)$ is a lattice point representation for $\mathfrak S_w$. Note that if $A_v$ and $B_v+(0,m)$ are connected by nonintersecting lattice paths, then all paths must start with $m$ upsteps by compactness. It follows that $(A_v, B_v+(0,m))$ represents the polynomial $(x_1\cdots x_m)^{n-m} \cdot \mathfrak S_v(x_{m+1}, \dots, x_{m+n}) = \mathfrak S_{1_m \ominus v}$ as in Proposition~\ref{prop:schubertsum}. Also $(A_u+(n-m,0), B_u+(n-m,0))$ represents $\mathfrak S_u$ as before. Since there are no directed paths from $A_v$ to $B_u+(n-m,0)$, Proposition~\ref{prop:product} implies that $(A_w, B_w)$ represents the product $\mathfrak S_u \cdot \mathfrak S_{1_m \ominus v}$, which equals $\mathfrak S_w$ by Proposition~\ref{prop:schubertsum}.
	
	Suppose instead that $w = u \oplus v$. Since $w$ avoids $1324$, $u$ must avoid $132$ and $v$ must avoid $213$. Hence $v' = 1_m \oplus v$ must also avoid $213$. We can then construct a lattice path representation $(A_{v'}, B_{v'})$ of $v'$ using Lemma~\ref{lemma:213}. Note that the code of $w_0^{(n)}v'w_0^{(n)} = w_0^{(n-m)}vw_0^{(n-m)} \oplus 1_m$ ends with $m$ zeroes. Thus $(0,0), (0,1),\dots, (0,m-1) \in B_{v'}$. By Proposition~\ref{prop:delete}, it follows that $(A'_{v'}, B'_{v'})$ is also a lattice path representation of $v'$, where
	\begin{align*}
		A'_{v'} = A_{v'} \setminus \{(0,0), (1,0), \dots, (m-1,0)\},\\
		B'_{v'} = B_{v'} \setminus \{(0,0), (0,1), \dots, (0,m-1)\}.
	\end{align*}
	(In fact, if $B_v$ is constructed for $v$ using Lemma~\ref{lemma:213}, then $B'_{v'}$ is the translation $B_v + (0,m)$.)
	
	Now consider the lattice path representation $(A_u, B_u)$ as constructed by Lemma~\ref{lemma:132}. If $(b_i,i-1) \in B_u$, then $b_i \leq m-i$ by the definition of the code of $u$. Hence there does not exist a directed path from any point of $A'_{v'}$ to any point in $B_u$. By Proposition~\ref{prop:product}, it follows that $(A_w, B_w) = (A_u \cup A'_{v'}, B_u \cup B'_{v'})$ is a lattice path representation of $\mathfrak S_u \cdot \mathfrak S_{1_m \oplus v}$, which equals $\mathfrak S_w$ by Proposition~\ref{prop:schubertsum}.
\end{proof}

\begin{Ex} \label{ex:87321564}
	Let $w = 87321564 = 21 \ominus v$, where $v = 321564$.
	
	Since $v = 321 \oplus 231$, $\mathfrak S_{v} = \mathfrak S_{321} \cdot \mathfrak S_{123564}$. Now $321$ is $132$-avoiding with code $(2,1,0)$, while $123564$ is $213$-avoiding with
	\[c(w_0 \cdot 123564 \cdot w_0) = c(312456) = (2,0,0,0,0,0).\] Reversing this second code and combining with the first gives $(2,1,0,0,0,2)$, so $321564$ has compact lattice path representation $(A_v,B_v)$ (up to sign) with 
	\begin{align*}
		A_v&=\{(0,0),(1,0), (2,0), (3,0), (4,0), (5,0)\},\\
		B_v&=\{(2,0),(1,1),(0,2),(0,3),(0,4),(2,5)\}.
	\end{align*}

	Now $\mathfrak S_w = \mathfrak S_{21} \cdot (x_1\cdots x_6)^2 \cdot \mathfrak S_v$. Shifting $B_v$ up by $2$ and placing a representation for $\mathfrak S_{21}$ to its right gives
	\begin{align*}
		A_w &= \{(0,0), (1,0), (2,0), (3,0), (4,0), (5,0), (6,0), (7,0)\},\\
		B_w &= \{(7,0), (6,1), (2,2),(1,3),(0,4),(0,5),(0,6),(2,7)\}.
	\end{align*}
	Then $(A_w, B_w)$ is a compact lattice path representation (up to sign) for $\mathfrak S_w$. To see how the representations for $\mathfrak S_{21}$, $\mathfrak S_{321}$, and $\mathfrak S_{231}$ fit together geometrically to give the representation for $\mathfrak S_w$, see Figure~\ref{fig:87321564}.
\end{Ex}

\begin{figure}
	\begin{tikzpicture}[scale=.7]
		\foreach \x in {0,...,7}
		\foreach \y in {0,...,6}
		\draw (\x,\y)--(\x,\y+1);
		\foreach \x in {1,...,7}
		\foreach \y in {0,...,6}
		\draw (\x,\y)--(\x-1,\y+1);
		\foreach \x/\y in {7/0,6/1,2/2,1/3,0/4,0/5,0/6,2/7}
			\node[w] at (\x,\y){};
		\draw[dashed] (5.7,-.3) rectangle (7.3,1.3);
		\draw[dashed] (-.3,1.7) rectangle (2.3,4.3);
		\draw[dashed] (-.3,4.7) rectangle (2.3,7.3);
	\end{tikzpicture}
	\caption{The set $B_w$ for a compact lattice path representation for $\mathfrak S_{w}$, where $w=87321564 = 21 \ominus (321 \oplus 231)$. The dashed squares from bottom to top are translations of $B$-sets for $21$, $321$, and $231$.}
	\label{fig:87321564}
\end{figure}
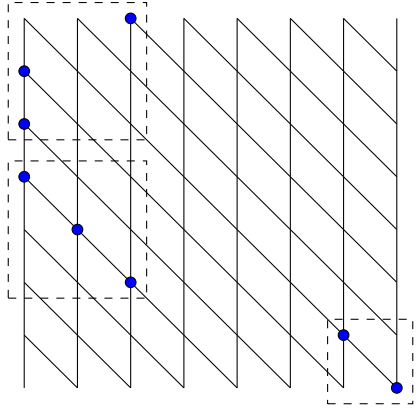

\subsection{Lowering points}

Given a lattice path representation, one can use Proposition~\ref{prop:delete} to remove lattice points at height $0$ and Proposition~\ref{prop:drop} to shift lattice points downward into empty rows, thereby generating additional representations. In this section, we will use these two operations on the collection of compact lattice path representations to construct representations for a large class of Schubert polynomials.

Application of these two operations can be described succinctly in the following way.
\begin{Def}
	We say a permutation $v \in S_n$ is a \emph{lowering permutation} if $v$ satisfies
	\[v^{-1}(1) > v^{-1}(2) > \dots > v^{-1}(k) = 1 < v^{-1}(k+1) < \cdots < v^{-1}(n)\]
	for some integer $k$. In other words, in one-line notation, $v$ contains $k(k-1) \cdots 21$ and $k(k+1)\cdots (n-1)n$ as subsequences.
\end{Def}
Equivalently, $v$ avoids the patterns $132$ and $312$. If we let $p_i = v^{-1}(i)$ for $1 \leq i \leq k$, then $v$ has the reduced expression \[v = (s_1s_2 \cdots s_{p_1-1}) \cdot (s_1s_2 \cdots s_{p_2-1}) \cdots (s_1s_2 \cdots s_{p_{k-1}-1}).\] Put another way, the effect of multiplying a permutation $u$ on the right by $v$ is to shuffle $u_k \cdots u_1$ and $u_{k+1} \cdots u_n$ by placing $u_k, \dots, u_1$ in positions $p_k, \dots, p_1$. 

The significance of these permutations to our current study lies in the following proposition.

\begin{Prop} \label{prop:lowering}
	Let $(A,B)$ be a lattice path representation of a polynomial $F$ of the form
	\begin{align*}
	A &= \{(a_1,0), (a_2,0), \dots, (a_n,0)\},\\
	B &= \{(b_1, 0), (b_2, 1), \dots (b_n, n-1)\}.
	\end{align*}
	Suppose further that $v$ is a lowering permutation with $v_1 = k$, and that $a_i=b_i$ for $i=1, \dots, k$. Then $\del_{v^{-1}}F$ has lattice path representation $(A',B')$, where
	\begin{align*}
		A' &= \{(a_{k+1},0), (a_{k+2},0), \dots, (a_n,0)\},\\
		B' &= \{(b_{k+1}, v^{-1}(k+1)-1), (b_{k+2}, v^{-1}(k+2)-1), \dots, (b_n, v^{-1}(n)-1)\}.
	\end{align*}
\end{Prop}
\begin{proof}
	Let $p_i = v^{-1}(i)$. By Proposition~\ref{prop:delete}, removing $(a_1,0) = (b_1,0)$ from $A$ and $B$ yields a lattice path representation for $F$. Then by Proposition~\ref{prop:drop}, lowering each of the points $(b_i,i-1)$ to $(b_i,i-2)$ for $i = 2, \dots, p_1$ yields a lattice path representation for $\del_{p_1-1} \cdots \del_2\del_1 F$ with points at heights $\{0, 1, \dots, n-1\} \setminus \{p_1-1\}$.
	
	We can then repeat this process by removing $(a_2,0) = (b_2,0)$ from both $A$ and $B$ and then lowering the points $(b_i, i-2)$ to $(b_i, i-3)$ for $i = 3, \dots, p_2$, giving a lattice path representation for $(\del_{p_2-1} \cdots \del_2\del_1)(\del_{p_1-1} \cdots \del_2\del_1)F$ with points at heights $\{0, 1, \dots, n-1\} \setminus \{p_1-1,p_2-1\}$. Continuing in this manner, we arrive at a lattice path representation for $\del_{v^{-1}} F$ with points at heights $\{0, 1, \dots, n-1\} \setminus \{p_1-1,\dots, p_k-1\} = \{v^{-1}(k+1)-1, \dots, v^{-1}(n)-1\}$, as desired.
\end{proof}
For an illustration, see Figure~\ref{fig:lowering} as well as Example~\ref{ex:lowering} below.

Note that any compact lattice path representation (up to reordering the elements of $A$ and $B$) has the form required in Proposition~\ref{prop:lowering}. Therefore, combining Proposition~\ref{prop:lowering} with Theorem~\ref{thm:compact} gives the following result.

\begin{Th} \label{thm:uv}
	Let $u,v \in S_n$ be permutations such that $u$ avoids the patterns $1324$, $2413$, and $3142$, $v$ avoids the patterns $132$ and $312$, and $\ell(uv) = \ell(u)-\ell(v)$. Then $\mathfrak S_{uv}$ has a proper lattice path representation.
\end{Th}
\begin{proof}
	By Theorem~\ref{thm:compact}, $\mathfrak S_u$ has a compact lattice path representation. Since $v$ is a lowering permutation, Proposition~\ref{prop:lowering} implies that $\partial_{v^{-1}} \mathfrak S_u$ has a proper lattice path representation. The length condition then implies $\partial_{v^{-1}} \mathfrak S_u = \mathfrak S_{uv}$.
\end{proof}

The following proposition gives an explicit description of when the length condition in Theorem~\ref{thm:uv} holds.

\begin{Prop} \label{prop:length}
	Let $u, v \in S_n$ be permutations such that $v$ is a lowering permutation. Suppose $v_1 = k$ and let $p_i = v^{-1}(i)$. If $w=uv$, then $\ell(w) = \ell(u)-\ell(v)$ if and only if $w_{p_i}$ is a left-to-right maximum of $w$ for all $i=1, \dots, k$.
\end{Prop}
\begin{proof}
	The effect of multiplying $u$ by 
	\[v = (s_1s_2 \cdots s_{p_1-1}) \cdot (s_1s_2 \cdots s_{p_2-1}) \cdots (s_1s_2 \cdots s_{p_{k-1}-1})\]
	is to shift $u_1$ to position $p_1$, then shift $u_2$ to position $p_2$, and so forth. The length condition will then be satisfied if while shifting $u_i$, it only moves past smaller letters. This occurs exactly when $w_{p_i}$ is a left-to-right maximum of $w$.
\end{proof}

Note that the $w_{p_i}$ need only be a subset of the left-to-right maxima of $w$, not the entire set of them.

\begin{Ex} \label{ex:lowering}
	Let $u = 87321564$ as in Example~\ref{ex:87321564}, and let $v = 34562718$, so that $p_1 = 7$, $p_2 = 5$, and $p_3=1$. Then $w=uv = 32157684$. Since $w_1=3$, $w_5=7$, and $w_7=8$ are left-to-right maxima, $\ell(w) = \ell(u)-\ell(v)$. 
	
	From Example~\ref{ex:87321564}, $\mathfrak S_u$ has compact lattice path representation
	\begin{align*}
		A_u &= \{(0,0), (1,0), (2,0), (3,0), (4,0), (5,0), (6,0), (7,0)\},\\
		B_u &= \{(7,0), (6,1), (2,2),(1,3),(0,4),(0,5),(0,6),(2,7)\}.
	\end{align*}
	By Proposition~\ref{prop:lowering}, $\mathfrak S_w$ then has lattice path representation
	\begin{align*}
		A_u &= \{(0,0), (1,0), (3,0), (4,0), (5,0)\},\\
		B_u &= \{(1,1),(0,2),(0,3),(0,5),(2,7)\}.
	\end{align*}
	See Figure~\ref{fig:lowering} for an illustration.
\end{Ex}

\begin{figure}
	\begin{tikzpicture}[scale=.7]
		\foreach \x in {0,...,7}
		\foreach \y in {0,...,6}
		\draw (\x,\y)--(\x,\y+1);
		\foreach \x in {1,...,7}
		\foreach \y in {0,...,6}
		\draw (\x,\y)--(\x-1,\y+1);
		\foreach \x in {0, ..., 6}
		\node[v] at (\x,0){};
		\foreach \x/\y in {6/1,2/2,1/3,0/4,0/5,0/6,2/7}
		\node[w] at (\x,\y){};
		\node[v, fill=violet] at (7,0){};
		\draw[->] (7.5,3.5)--(8.5,3.5);
		\begin{scope}[shift={(9,0)}]
		\foreach \x in {0,...,7}
		\foreach \y in {0,...,6}
		\draw (\x,\y)--(\x,\y+1);
		\foreach \x in {1,...,7}
		\foreach \y in {0,...,6}
		\draw (\x,\y)--(\x-1,\y+1);	
		\foreach \x in {0, 1, 3, 4, 5}
		\node[v] at (\x,0){};
		\foreach \x/\y in {1/1,0/2,0/3,0/5,2/7}
		\node[w] at (\x,\y){};
		\end{scope}
	\end{tikzpicture}
	\caption{Illustration of Proposition~\ref{prop:lowering}. On the left is a proper lattice path representation for $\mathfrak S_u$ (where $u=87321564$), and on the right is a lattice path representation for $\mathfrak S_{uv}$, where $v=34562718$.}
	\label{fig:lowering}
\end{figure}
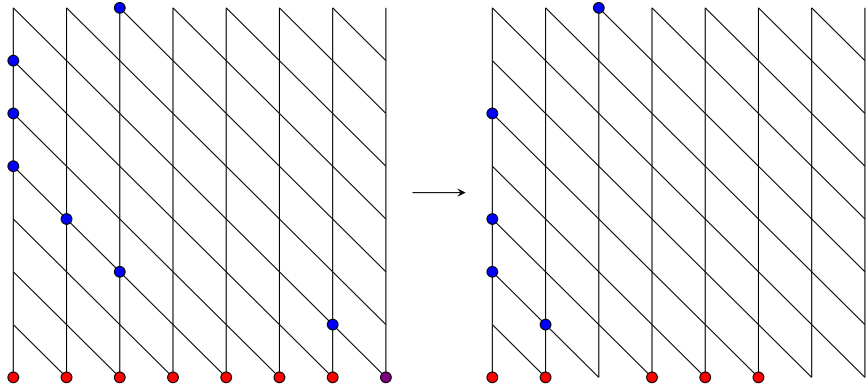

	As another illustrative example, we consider the case of $321$-avoiding permutations, whose Schubert polynomials are known to be flagged skew Schur polynomials \cite{BJS}.
	
\begin{Cor} \label{cor:321}
	Let $w \in S_n$ be a $321$-avoiding permutation. Let $\bar q_1 < \cdots < \bar q_{n-k}$ be the elements of $[n]$ that are \emph{not} left-to-right maxima of $w$, and let $\bar p_i = w^{-1}(\bar q_i)$. Then $\mathfrak S_w$ has lattice path representation $(A,B)$, where
	\begin{align*}
		A&= \{(\bar q_1-1,0), (\bar q_2-1,0), \dots, (\bar q_{n-k}-1,0)\},\\
		B&= \{(0, \bar p_1-1), (0, \bar p_2-1), \dots, (0, \bar p_{n-k}-1)\},
	\end{align*}
	and therefore
	\[\mathfrak S_w = \det (e_{\bar p_j - \bar q_i}^{(\bar p_j-1)})_{i,j=1}^k. \]
\end{Cor}

\begin{proof}	
	Let $w \in S_n$ have left-to-right maxima in positions $p_1 > p_2 > \dots > p_k=1$, and let $q_i = w_{p_i}$ be the values of these maxima (so that $q_1 > q_2 > \dots > q_k$).

	Since $w$ is $321$-avoiding, the letters $\bar q_1, \dots, \bar q_{n-k}$ must appear in increasing order in $w$, so $\bar p_1< \dots< \bar p_{n-k}$. Let $v$ be the lowering permutation with $v^{-1} = p_1 \cdots p_k \bar p_1 \dots \bar p_{n-k}$. If we let $u=wv^{-1}$, then $u=q_1 \cdots q_k\bar q_1 \cdots \bar q_{n-k}$ and $\ell(w) = \ell(u)-\ell(v)$ by Proposition~\ref{prop:length}.
	
	Now observe that $u$ is $132$-avoiding and $c(u)=(q_1-1, \dots, q_k-1, 0, \dots , 0)$, so by Lemma~\ref{lemma:132}, $\pm\mathfrak S_u$ has lattice path representation $(A',B')$, where
	\begin{align*}
		A'&= \{(n-1,0), (n-2,0), \dots, (0,0)\},\\
		B'&= \{(q_1-1,0), (q_2-1,1), \dots, (q_k-1,k-1), (0,k), \dots,  (0,n-1)\}.
	\end{align*}
	Applying Proposition~\ref{prop:lowering}, we find that $\mathfrak S_w$ has lattice path representation $(A,B)$, as desired. (The sign is easily seen to be positive.) The determinantal formula then follows from Proposition~\ref{edet}.
\end{proof}

Since Grassmannian permutations are special cases of $321$-avoiding permutations, Corollary~\ref{cor:321} specializes to a formula for Schur polynomials akin to the dual Jacobi-Trudi identity, as also shown in \cite{Kirillov, Winkel}. (Compare the following to Proposition~\ref{prop:jacobitrudi}.)

\begin{Cor} \label{cor:grassmannian}
	Let $\lambda$ be a partition with largest part $r$. Then the Schur polynomial $s_\lambda(x_1, \dots, x_n)$ is given by the determinant
	\[s_\lambda(x_1, \dots, x_n) = \det(e^{(n+j-1)}_{\lambda_i'+j-i})_{i,j=1}^r.\]
\end{Cor}
\begin{proof}
	The Schur polynomial $s_\lambda(x_1, \dots, x_n)$ is equal to the Schubert polynomial $\mathfrak S_w$, where $w \in S_{n+r}$ is the Grassmannian permutation $q_1q_2 \cdots q_n \bar q_1 \bar q_2 \cdots \bar q_r$, where $q_i = \lambda_{n+1-i}+i$ and $\bar q_i = n-\lambda'_i+i$. Since the left-to-right maxima are precisely $q_1, \dots, q_n$ and $\bar p_i = w^{-1}(\bar q_i)=n+i$, the result follows from Corollary~\ref{cor:321}.
\end{proof}

One can also deduce Corollary~\ref{cor:grassmannian} by interpreting the usual dual Jacobi-Trudi identity (Proposition~\ref{prop:jacobitrudi}) as a lattice point representation (albeit not a proper one) and applying Proposition~\ref{prop:pull} repeatedly to turn it into a proper representation.

\subsection{Pattern avoidance criterion}

In this section, we will give an explicit description of the permutations to which Theorem~\ref{thm:uv} applies via the following theorem.

\begin{Th} \label{thm:main}
	A permutation $w$ has a factorization of the form $w=uv$ as in Theorem~\ref{thm:uv} if and only if it avoids the following 13 patterns:
\[	51324, \quad
	15324, \quad
	52413, \quad
	25413, \quad
	53142, \quad
	35142, \quad
	31542,\]\[
	143265, \quad
	143625, \quad
	143652, \quad
	146352, \quad
	413265, \quad
	413625.\]
	Therefore, for any such permutation $w$, $\mathfrak S_w$ has a proper lattice path representation.
\end{Th}

Note that this theorem gives a sufficient, but not a necessary, condition for $\mathfrak S_w$ to have a proper lattice path representation. For further discussion, see \S5.

While the proof of the forward direction of Theorem~\ref{thm:main} will be relatively straightforward, for the reverse direction we will need to describe for each permutation $w$ avoiding the given 13 patterns how to construct the corresponding permutations $u$ and $v$. By Proposition~\ref{prop:length}, we will choose $v$ by choosing a certain subset of the left-to-right maxima of $w$. We will then verify that $u = wv^{-1}$ avoids $2413$, $3142$, and $1324$ as required.

Fix a permutation $w \in S_n$ that avoids the 13 patterns in Theorem~\ref{thm:main}. We construct a set $Q \subseteq [n]$ as follows. Consider the left-to-right maxima of $w$ from largest to smallest (i.e., from right to left). For each such $q$, add it to $Q$ unless $w$ has an occurrence of the pattern $1342$ consisting of letters $aqq'b$, where $q' \notin Q$. 

\begin{Ex}
	Let $w = 32157684$, which avoids the $13$ patterns in Theorem~\ref{thm:main}. The left-to-right maxima of $w$ are $8$, $7$, $5$, and $3$.
	\begin{itemize}
		\item We first add $8$ to $Q$ since it cannot be the second letter in a $1342$ pattern.
		\item Although $7$ is the second letter of several $1342$ patterns, the third letter in such patterns is always $8 \in Q$, so we add $7$ to $Q$.
		\item Now $5$ occurs in $1564$ and $6 \notin Q$, so we do not add $5$ to $Q$.
		\item Finally, we add $3$ to $Q$, so that $Q = \{3, 7, 8\}$.
	\end{itemize}
	The elements of $Q$ occur at positions $1$, $5$, and $7$. Note that if we let $v$ be the lowering permutation $34562718$, then the permutation $u = wv^{-1} = 87321564$ obtained by shifting the elements of $Q$ to the left in decreasing order is a $1324$-avoiding separable permutation.
\end{Ex}

We will also need some technical lemmas about the structure of the permutations in Theorem~\ref{thm:main}.

\begin{Lemma} \label{lemma:13542}
	Let $w$ be a permutation that avoids the 13 patterns in Theorem~\ref{thm:main}. If $w$ has a subsequence $abcde$ that forms a $13542$ pattern, then any letter that occurs between $b$ and $d$ in $w$ must be greater than $b$.
\end{Lemma} 
\begin{proof}
	Suppose $x$ lies between $b$ and $d$ in $w$. If $x < e$, then $w$ must contain either the $31542$ pattern $bxcde$ or the $35142$ pattern $bcxde$. If instead $e<x<b$, then $w$ contains either the $143652$ pattern $abxcde$ or the $146352$ pattern $abcxde$. Since all of these patterns are forbidden, we must have $x > b$.
\end{proof}

\begin{Lemma} \label{lemma:1342}
	Let $w$ be a permutation that avoids the 13 patterns in Theorem~\ref{thm:main}, and fix a left-to-right maximum $b \notin Q$. Let $c$ be the rightmost letter of $w$ such that $c \notin Q$ and $w$ contains a $1342$ pattern $abcd$. Then either $w$ contains a $13542$ pattern $abxcd$, or $w$ contains a $2413$ pattern $bcde$. 
\end{Lemma}
\begin{proof}
	Since $c \notin Q$, there are two possibilities.
	\begin{itemize}
		\item If $c$ is not a left-to-right maximum, then there must be a larger letter $x$ to its left. Since $b$ is a left-to-right maximum and $x > c > b$, $w$ must have the $13542$ pattern $abxcd$.
		\item If $c$ is a left-to-right maximum, then since $c \notin Q$, it must be part of a $1342$ pattern $fcge$ with $g \notin Q$. By our choice of $c$ to be rightmost, we must have that $g$ lies to the right of $d$ (or else $abgd$ would be a $1342$ pattern). Then:
		\begin{itemize}
			\item If $a < e < b$, then $abge$ would be a $1342$ pattern that contradicts our choice of $c$.
			\item If $e < a$, then $f$ cannot lie to the left of $b$ or else $fbge$ would be a $1342$ pattern that contradicts our choice of $c$. Hence $f$ has to lie to the right of $b$, but then $w$ would contain the $35142$ pattern $abfde$, which is a contradiction.
		\end{itemize}
		The only remaining possibility is that $e > b$, which implies that $w$ has the $2413$ pattern $bcde$, as desired. \qedhere
	\end{itemize}
\end{proof}

Using Lemmas~\ref{lemma:13542} and \ref{lemma:1342}, we can now prove most of the pattern conditions that we will need for Theorem~\ref{thm:main}.

\begin{Lemma} \label{lemma:qpatterns}
	Let $w$ be a permutation that avoids the 13 patterns in Theorem~\ref{thm:main}.
	\begin{enumerate}[(a)]
		 \item Suppose $w$ contains the $2413$ pattern $abcd$. Then $b \in Q$.
		 \item Suppose $w$ contains the $3142$ pattern $abcd$. Then $c \in Q$.
		 \item Suppose $w$ contains the $1324$ pattern $abcd$ with $d \notin Q$. Then $b \in Q$.
		 \item Suppose $w$ contains the $1342$ pattern $abcd$ with $c \notin Q$. Then $b \notin Q$.
	\end{enumerate}
\end{Lemma}
\begin{proof}
	For (a), note that $b$ must be a left-to-right maximum, for if it were not, then some letter to the left of $b$ would be greater than $b$, which would cause $w$ to contain a forbidden $52413$ or $25413$ pattern.
	
	Suppose the claim does not hold, and let us take $b$ to be the rightmost left-to-right maximum in a $2413$ pattern $abcd$ with $b \notin Q$. By Lemma~\ref{lemma:1342}, either $w$ has some $2413$ pattern $befg$ with $e \notin Q$, which contradicts our choice of $b$, or $b$ contains a $13542$ pattern $ebfgh$. In the latter case, by Lemma~\ref{lemma:13542}, since $c$ and $d$ are both less than $b$, they must lie to the right of $g$. But then $w$ contains the forbidden $25413$ pattern $afgcd$, completing the proof of (a).
	
	
	Note that (a) implies that the second possibility in Lemma~\ref{lemma:1342} can never hold. In other words, any left-to-right maximum that does not lie in $Q$ must appear second in a $13542$ pattern.
	
	For (b), note that $c$ must be a left-to-right maximum or else $w$ would contain a forbidden $53142$, $35142$, or $31542$ pattern. Suppose $c \not\in Q$. Then by Lemma~\ref{lemma:1342} (as per the discussion above) there exists a $13542$ pattern $ecfgh$. By Lemma~\ref{lemma:13542}, $d < c$ cannot lie between $c$ and $g$, so $d$ must lie to the right of $g$. But then $abfgd$ is a forbidden $31542$ pattern in $w$. So $c \in Q$.
	
	For (c), for a fixed $b$, let us choose $d \notin Q$ to be rightmost. If $d$ were a left-to-right maximum, then there would have to be a $1342$ pattern $edfg$ with $f \notin Q$. But then the $1324$ pattern $abcf$ would contradict the choice of $d$. Hence $d$ is not a left-to-right maximum. Therefore, there exists some $h > d$ to the left of $d$. If $h$ lies to the left of $b$, then $w$ would either contain the $51324$ pattern $habcd$ or the $15324$ pattern $ahbcd$, which are both forbidden. Thus $h$ lies to the right of $b$ (and to the left of $d$).
	
	Suppose $b$ is not a left-to-right maximum. Then there exists some $i > b$ to the left of $b$. But we cannot have $i > d$ for then $w$ would contain $iabcd$ or $aibcd$, which would be a $51324$ or $15324$ pattern, nor can we have $i < d$ for then $w$ would contain one of $iabhcd$, $iabchd$, $aibhcd$, or $aibchd$, which would be a $413625$, $413265$, $143625$, or $143265$ pattern. Thus $b$ must be a left-to-right maximum.
	
	Now suppose for the sake of contradiction that $b \notin Q$. By Lemma~\ref{lemma:1342}, there exists a $13542$ pattern $jbklm$. By Lemma~\ref{lemma:13542}, $c < b$ cannot appear between $b$ and $l$, so it must appear after $l$. If $d < l$, then $bklcd$ would be a forbidden $25413$ pattern. If $l < d< k$, then $aklcd$ would be a forbidden $15324$ pattern. Hence $d > k$.
	
	Recall that $h > d$ lies to the right of $b$. If $h$ lies to the left of $l$, then $ahlcd$ would be a forbidden $15324$ pattern. Then $h$ must lie to the right of $l$, but now $w$ must contain either the $143265$ pattern $aklchd$ or the $143625$ pattern $aklhcd$, which are forbidden. It follows that we must have $b \in Q$, as desired.
	
	Finally, (d) follows immediately from the construction of $Q$.
\end{proof}

It is now straightforward to deduce our main result.

\begin{proof}[Proof of Theorem~\ref{thm:main}]
	We first verify that any permutation $w$ with a factorization $w=uv$ as in Theorem~\ref{thm:uv} must avoid the given $13$ patterns. Note that if $w'$ is a pattern of $w$, then there exist patterns $u'$ of $u$ and $v'$ of $v$ such that $w'=u'v'$. Any pattern $v'$ contained in the lowering permutation $v$ is again a lowering permutation. By Proposition~\ref{prop:length}, the length condition $\ell(w)=\ell(u)-\ell(v)$ implies that multiplying $u$ by $v$ has the effect of shifting the first $k$ letters in $u$ to become left-to-right maxima of $w$. But any left-to-right maximum of $w$ chosen to appear in $w'$ will still be a left-to-right maximum. It follows that $\ell(w') = \ell(u')-\ell(v')$, so $w'$ must also satisfy the conditions of Theorem~\ref{thm:uv}.
	
	Therefore, we need only verify that none of the $13$ patterns $w'$ have such a factorization $u'v'$.	To see this, observe that each pattern other than $143652$ and $146352$ contains a $1324$, $2413$, or $3142$ pattern that does not involve any left-to-right maxima except for possibly the first letter. Since these would necessarily remain in the same order in $u'$, $u'$ cannot avoid these three patterns. For the last two patterns $143652$ and $146352$, depending on whether the left-to-right maximum $4$ is moved, $u'$ must contain either the $1324$-pattern $1435$ or the $3142$-pattern $4152$.
	
	For the reverse direction, we need to verify that any permutation $w$ that avoids the given $13$ patterns has the requisite factorization $w=uv$. Defining the set $Q$ as described, let $u$ be the permutation obtained from $w$ by shifting the elements of $Q$ to the left and placing them in decreasing order, so that $u=wv^{-1}$ for some lowering permutation $v$ with $\ell(w)=\ell(u)-\ell(v)$ as in Proposition~\ref{prop:length}. If $u$ were to contain one of the patterns $2413$, $3142$, or $1324$, then there are only four possibilities for how these letters could be ordered in $w$:
	\begin{enumerate}[(a)]
		\item $w$ contains the $2413$ pattern $abcd$ and $b \notin Q$, so that $abcd$ occurs in $u$;
		\item $w$ contains the $3142$ pattern $abcd$ and $c \notin Q$, so that $abcd$ occurs in $u$;
		\item $w$ contains the $1324$ pattern $abcd$ and $b,d \notin Q$, so that $abcd$ occurs in $u$;
		\item $w$ contains the $1342$ pattern $abcd$ with $b \in Q$ and $c \notin Q$, so that the $3142$ pattern $bacd$ occurs in $u$.
	\end{enumerate}
	However, all of these are impossible by Lemma~\ref{lemma:qpatterns}, which completes the proof.	
\end{proof}

\section{Conclusion}

Although Theorem~\ref{thm:main} gives a determinantal formula for a wide class of Schubert polynomials, the precise characterization of which Schubert polynomials admit such a formula remains open.

\begin{?}
	For which permutations $w \in S_\infty$ does $\mathfrak S_w$ admit a proper lattice path representation (and hence a determinantal formula for its SEM expansion)? Is the set of such permutations closed under pattern containment?
\end{?}

We note in particular that the condition in Theorem~\ref{thm:main} is sufficient but not necessary. For example, although $413625$ is a forbidden pattern,
\[\mathfrak S_{413625} = \left|\begin{matrix}
	e_1^{(1)}&e_2^{(2)}&0&0\\
	e_0^{(1)}&e_1^{(2)}&e_4^{(4)}&e_5^{(5)}\\
	0&e_0^{(2)}&e_3^{(4)}&e_4^{(5)}\\
	0&0&e_0^{(4)}&e_1^{(5)}
\end{matrix}\right|\]
has the proper lattice path representation shown in Figure~\ref{fig:413265}. From this, one can then use Proposition~\ref{prop:slide} to derive representations for $\mathfrak S_{413265}$, $\mathfrak S_{143625}$, and $\mathfrak S_{143265}$.
(The Schubert polynomials for the remaining nine forbidden patterns, including all of the ones of length $5$, do not have proper lattice path representations.)

\begin{figure}
	\begin{tikzpicture} [scale=.7]
		\foreach \x in {0,...,5}
		\foreach \y in {0,...,4}
		\draw[dashed] (\x,\y)--(\x,\y+1);
		\foreach \x in {1,...,5}
		\foreach \y in {0,...,4}
		\draw[dashed] (\x,\y)--(\x-1,\y+1);
		\node[v] (a1) at (0,0){};
		\node[v] (a2) at (1,0){};
		\node[v] (a3) at (2,0){};
		\node[v] (a4) at (5,0){};
		\node[w] (b1) at (0,1){};
		\node[w] (b2) at (0,2){};
		\node[w] (b3) at (1,4){};
		\node[w] (b4) at (1,5){};
		\draw[thick] (a1)--(b1);
		\draw[thick] (a2)--(1,1)--(b2);
		\draw[thick] (1,2)--(1,3)--(b3)--(2,3)--(2,2)--(2,1)--(a3) (2,1)--(1,2) (2,2)--(1,3);
		\draw[thick] (a4)--(4,1)--(3,2)--(2,3)--(2,4)--(b4) (2,4)--(3,3)--(4,2)--(5,1)--(a4) (4,1)--(4,2) (3,2)--(3,3);
		
	\end{tikzpicture}
	\caption{Lattice path representation for $\mathfrak S_{413625}$. (Edges that appear in at least one collection of nonintersecting lattice paths are solid.)}
	\label{fig:413265}
\end{figure}
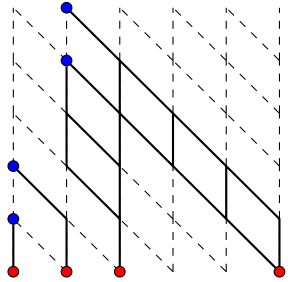


Recall that any polynomial with a proper lattice path representation also has the property that its SEM expansion only has coefficients of absolute value at most $1$. One can then ask similar questions about the class of Schubert polynomials satisfying this weaker property. (See Winkel \cite{Winkel} for some discussion, as well as \cite{BP,FMS} for some similar studies.)

\begin{?}
	For which permutations $w \in S_\infty$ does the SEM expansion of $\mathfrak S_w$ have only coefficients of absolute value at most $1$? Is the set of such permutations closed under pattern containment?
\end{?}

Our proof of Theorem~\ref{thm:main} is algebraic as opposed to combinatorial. A bijective proof certainly exists for certain subclasses of permutations (for instance, Grassmannian permutations), and to some extent one can use the operations of \S3 to generate bijections for other cases covered by Theorem~\ref{thm:main}. However, it is unclear whether a uniform bijection exists in general, particularly in cases not covered by Theorem~\ref{thm:main}.

\begin{?}
	When $\mathfrak S_w$ has a proper lattice path representation, is there a natural bijection between the corresponding collections of nonintersecting lattice paths and other known combinatorial interpretations for $\mathfrak S_w$ (such as reduced pipe dreams)?
\end{?}

%

\bibliographystyle{acm}
\bibliography{citations}

\end{document}